\documentclass{amsart}

\usepackage{graphicx}
\usepackage{graphics}
\usepackage{amsmath,amsthm,verbatim,amssymb,amsfonts,amscd, graphicx}

\usepackage{tikz}
\usepackage{parskip}
\usepackage{float}
\newtheorem{theorem}{Theorem}[section]
\newtheorem{lemma}[theorem]{Lemma}
\newtheorem{proposition}{Proposition}[section]
\theoremstyle{definition}
\newtheorem{definition}[theorem]{Definition}

\newtheorem{corollary}{Corollary}
\newtheorem{question}{Question}
\theoremstyle{remark}
\newtheorem{remark}[theorem]{Remark}
\newtheorem{conjecture}{Conjecture}
\numberwithin{equation}{section}

\begin{document}

\title{The prime index function}

\author{Theophilus Agama}
\address{Department of Mathematics, African Institute for Mathematical science, Ghana
}
\email{theophilus@aims.edu.gh/emperordagama@yahoo.com}


\subjclass[2000]{Primary 54C40, 14E20; Secondary 46E25, 20C20}

\date{\today}



\footnote{
\par 
.}%
.

\maketitle

\begin{abstract}
In this paper we introduce the prime index function \begin{align}\iota(n)=(-1)^{\pi(n)},\nonumber
\end{align} where $\pi(n)$ is the prime counting function. We study some elementary properties and theories associated with the partial sums of this function given by\begin{align}\xi(x):=\sum \limits_{n\leq x}\iota(n).\nonumber
\end{align}We show that a prime $p>2$ is a twin prime if and only if $\xi(p)=\xi(p+2)$. We also relate the prime index function to Cramer's conjecture by showing that \begin{align}|\xi(p_{n+1})-\xi(p_n)|+2=p_{n+1}-p_n.\nonumber
\end{align}That is, Cramer's conjecture can stated as \begin{align}\xi(p_{n+1})-\xi(p_n)\ll (\log p_n)^2.\nonumber
\end{align}This reduces the problem to obtaining very good estimates of the second prime index function.
\end{abstract}

\section{Introduction and motivation}
The prime counting function is a very important and useful function in number theory and the whole of mathematics. It is connected to many open problems in mathematics, such as the Riemman hypothesis \cite{May}. In the following sequel we introduce the prime index function $\iota(n)$, an arithmetic function which is neither additive nor multiplication. It can be considered to be of the same class with the Liouville $\lambda(n)$ function. Rather like the Liouville function defined on the prime factors of the integers, the prime index function is defined on the number of prime less than a fixed integer. It is given by \begin{align}\iota(n)=(-1)^{\pi(n)},\nonumber
\end{align}where $\pi(n)$ is the prime counting function. Given the chaotic behaviour of the prime index-function makes it an intractable function to study. Hence we introduce as well the second prime-index function given by \begin{align}\xi(x):=\sum \limits_{n\leq x}\iota(n).\nonumber
\end{align}By introducing the concept of oscillation on the second prime-index function $\xi$, we relate results of primes in short and long intervals to oscillations. It turns out that the following, which can be considered as a sibling of Bertrand's postulate, is true:

\begin{theorem}
Let $x,y\in \mathbb{R}$. If $\xi(x)=\xi(y)$, then there exist at least a prime in the interval $(x,y]$.
\end{theorem}
In the spirit of understanding the twin prime conjecture, we obtain the following weaker result:

\begin{theorem}
There are infinitely many points $y\in \mathbb{R}$ such that $\xi(y-s)=\xi(y+s)$ for some $s\in \mathbb{N}$. In particular, there are infinitely many points of oscillations with period $s$.
\end{theorem}

\section{The prime index function}

\begin{definition}\label{first}
Let $n\geq 1$, then we set \begin{align}\iota(n)=(-1)^{\pi(n)},\nonumber 
\end{align}
\end{definition}where $\pi(n)$ is the prime counting function.
\bigskip

The prime index function is an extremely usefull function. It is basically the sequence \begin{align}\iota:\mathrm{N}\longrightarrow \{1,-1\}.\nonumber
\end{align}It is somewhat an intractable but very interesting function when we take their partial sums. Below is a table for the distribution of the prime-index function:

\begin{table}[ht]
\caption{}\label{eqtable}
\renewcommand\arraystretch{1.5}
\noindent\[
\begin{array}{|c|c|c|c|c|c|c|c|c|c|c|c|c|c|c|c|c|c|c|}
\hline
Values~of n & 1 & 2 & 3 & 4 & 5 & 6 & 7 & 8 & 9 & 10 & 11 & 12 & 13 & 14 & 15 & 16 & 17 & 18\\
\hline
\iota(n)& 1 & -1 & 1 & 1 & -1 & -1 & 1 & 1 & 1 & 1 & -1 & -1 & 1 & 1 & 1 & 1 & -1 & -1\\
\hline 
\end{array}
\]
\end{table}

\begin{definition}
Let $x\geq 1$, then we set \begin{align}\xi(x):=\sum \limits_{n\leq x}\iota(n).\nonumber
\end{align}
\end{definition}
\bigskip

Next we make a leap by understanding the distribution of the partial sums of the prime index function $\xi(x)$. Below is a table that gives the distribution of the first eighteen values of $\xi(x)$.

\section{Distribution of the second prime index function $\xi(x)$}
In this section we give the distribution of the second prime-index function for the first eighteen values of the integers. We also establish a relationship between the values of this function and the theory of gaps between primes. 

\begin{table}[ht]
\caption{}\label{eqtable}
\renewcommand\arraystretch{1.5}
\noindent\[
\begin{array}{|c|c|c|c|c|c|c|c|c|c|c|c|c|c|c|c|c|c|c|}
\hline
Values~of x & 1 & 2 & 3 & 4 & 5 & 6 & 7 & 8 & 9 & 10 & 11 & 12 & 13 & 14 & 15 & 16 & 17 & 18\\
\hline
\xi(x) & 1 & 0 & 1 & 2 & 1 & 0 & 1 & 2 & 3 & 4 & 3 & 2 & 3 & 4 & 5 & 6 & 5 & 4\\
\hline 
\end{array}
\]
\end{table}

By observing critically the table of distribution of the prime-index function, we observe that the second prime index function can be very high in absolute terms. This is due to the fact that there are arbitrarily large gap between primes.  Occasionally it could return to zero, which is the smallest it can be in absolute terms, unless it is rescued by some prime. Despite this, we believe the following to be true:

\begin{conjecture}
Let $x\in \mathbb{R}$, then there exist some $N_0>0$ such that \begin{align}\mathrm{sup}_{x\geq 1}\bigg\{\sum \limits_{n\leq x}(-1)^{\pi(n)}\bigg\}\geq N\nonumber
\end{align}for $N\geq N_0$.
\end{conjecture}
\bigskip

Overall the second prime-index function is neither decreasing nor increasing. However we could be certain about the monotonicity behaviour of the second prime-index function if we restrict to the subsequence of the integers such as the primes. We ask the following question in that direction.

\begin{question}
Are there infinitely many primes $p,q$ with $p<q$ such that $\xi(p) \leq \xi(q)$?
\end{question}
\bigskip

Again by observing the distribution of the second prime-index function $\xi$, we notice that occasionally it behaves on some points like $\xi(x)\approx \frac{\pi(x)}{2}$. So we ask the broader question as follows:

\begin{question}
Are there infinitely many real numbers $x$ such that \begin{align}\xi(x)\approx \frac{\pi(x)}{2}?\nonumber
\end{align}
\end{question}

\section{Relationship with the theory of prime gaps}
In the spirit of relating the second prime-index function to the theory of prime gaps, we introduce the notion of oscillation on $\xi$.
\begin{definition}\label{oscillation}
Let $y\geq 2$, then we say $\xi$ oscillates at $y$ with period $M$ if there exist some $r_i\in \mathbb{R}^{+}$ such that \begin{align}\xi(x-r_i)=\xi(x+r_i),\nonumber
\end{align}where $M=\mathrm{min}\{r_i\}$.
\end{definition}
\bigskip
Geometrically, we could think of the notion of oscillation of $\xi$ at a point $x$ with period $t$ as the change in the gradient of $\xi$ of the line connecting the points $(x,\xi(x)$, $(x-t,\xi(x-t))$ and $(x,\xi(x))$, $(x+t,\xi(x+t))$ upto signs. Suppose $x$ is a point of oscillation of $\xi$ with period $t$. Then it follows that $\xi(x-t)=\xi(x+t)$. Thus we observe that the gradient of the line joining the point $(x,\xi(x))$ and $(x+t, \xi(x+t))$ is given by \begin{align}\mathcal{G}_{x,x+t}=\frac{\xi(x+t)-\xi(x)}{t}.\nonumber  
\end{align}Similarly the gradient of the line joining the point $(x,\xi(x))$ and $(x-t, \xi(x-t))$ is given by \begin{align}\mathcal{G}_{x-t,x}&=\frac{\xi(x)-\xi(x-t)}{t}\nonumber \\&=\frac{\xi(x)-\xi(x+t)}{t}\nonumber \\&=-\bigg(\frac{\xi(x+t)-\xi(x)}{t}\bigg)\nonumber \\&=-\mathcal{G}_{x,x+t}.\nonumber
\end{align}
\begin{remark}
Next we prove a result that suggests that there are infinitely many points of oscillations of the second prime-index function.
\end{remark}

\begin{theorem}\label{infinity}
There are infinitely many points $y\in \mathbb{R}$ such that $\xi(y-s)=\xi(y+s)$ for some $s\in \mathbb{N}$. In particular, there are infinitely many points of oscillations.
\end{theorem}

\begin{proof}
Suppose on the contrary that there are finitely many points of oscillations of $\xi$. Name them $y_1< y_2< \cdots <y_n$. It follows that there exist some least $s\in \mathbb{N}$ such that $\xi(y_n-s)=\xi(y_n+s)$. Without loss of generality, let $\xi(y_n)<\xi(y_{n}+s)$. Since any infinite sequence of the form $y_{r}<y_{r+1}<y_{r+2}<\cdots $ for $r>n$ are not points of oscillation of $\xi$, it follows that $\xi(y_{n}+s)<\xi(y_{n}+s+1)<\xi(y_{n}+s+2)<\xi(y_{n}+s+3)<\cdots $. It follows that $\pi(y_{n}+s)=\pi(y_{n}+s+1)=\pi(y_{n}+s+2)=\pi(y_{n}+s+3)=\cdots \pi(y_{n}+j)=\cdots$. This contradicts the infinitude of primes, thereby ending the proof.
\end{proof}

\begin{remark}
Next we prove a result that is redolent of  Betrand's postulate \cite{hildebrand2005introduction}, which asserts that we can find at least one prime in any interval of the form $(x,2x]$ for any $x\geq 1$. The result below can be seen as a variant and a strengthening of the postulate. 
\end{remark}

\begin{theorem}\label{bertrand}
Let $x,y\in \mathbb{R}$ with $x<y$. If $\xi(x)=\xi(y)$, then there exist at least a prime in the interval $(x,y]$.
\end{theorem}

\begin{proof}
Suppose $x,y\in \mathbb{R}$ and let $\xi(x)=\xi(y)$. Then it follows that \begin{align}\sum \limits_{x<n\leq y}(-1)^{\pi(n)}=0.\nonumber
\end{align}It follows that the sequence $\{1,-1\}$ are evenly distributed for points on $(x,y]$ and,  hence the parity of the prime counting function $\pi(n)$ for points in $(x,y]$ are equidistributed. The result follows immediately from this fact. 
\end{proof}

\begin{proposition}
Let $x$ be a point of oscillation of $\xi$ with period $t\geq 1$. If the interval $(x-t,x+t)$ contains no prime, then $x+t$ must necessarily be prime.
\end{proposition}

\begin{proof}
Let $x$ be a point of oscillation of $\xi$ with period $t$. Then it follows from Definition \ref{oscillation} that $\xi(x-t)=\xi(x+t)$. Then it follows  by Theorem \ref{bertrand} that the interval $(x-t,x+t]$ contains at least one prime. Since the interval $(x-t,x+t)$ contains no prime, it follows that $x+t$ must be prime and the result follows immediately.
\end{proof}

\begin{remark}
Next we strengthened Theorem \ref{bertrand} by imposing some extra conditions for points of coincidences of the second prime-index function $\xi$.
\end{remark}

\begin{theorem}
Let $[a,b]\subset \mathbb{R}$. If there exist finitely many points $x_1,x_2,\ldots, x_n \in [a,b]$ such that $\xi(x_1)=\xi(x_2)=\ldots =\xi(x_n)$, then the interval $[a,b]$ contains finitely many primes.
\end{theorem}

\begin{proof}
Let $[a,b]\subset \mathbb{R}$. Suppose $x_1,x_2,\ldots x_n\in [a,b]$ and $x_1<x_2<\ldots<x_n$ such that $\xi(x_1)=\xi(x_2)=\ldots =\xi(x_n)$. Then it follows from Theorem \ref{bertrand} that there is at least a prime in each of the intervals $(x_1,x_2], (x_2,x_3], \ldots (x_{n-1},x_n]$, and the result follows immediately.
\end{proof}
\bigskip

We highlight the next result, which reinforces the very fact that more and more coincidences of the primes on the second prime-index function has a profound connection with primes in an arithmetic progression.

\begin{lemma}\label{key lemma}
Let $p_1<p_2<\ldots <p_n$ be a sequence of consecutive primes. If $\xi(p_1)=\xi(p_2)=\ldots =\xi(p_n)$, then $p_1<p_2<\ldots <p_n$ must be an arithmetic progression of common difference $2$.
\end{lemma}

\begin{proof}
Consider the sequence of consecutive primes $p_1<p_2<\ldots p_n$ such that $\xi(p_1)=\xi(p_2)=\ldots =\xi(p_n)$. Then it follows from Theorem \ref{bertrand} that the intervals $(p_1,p_2], (p_2,p_3],\ldots (p_{n-1},p_n]$ each contains at least one prime. Since $p_1<p_2<\cdots <p_n$ are consecutive primes, it follows that there are no primes in each of the open intervals $(p_1,p_2), (p_2,p_3), \ldots (p_{n-1},p_n)$. Since $\xi(p_1)=\xi(p_2)=\ldots =\xi(p_n)$, it follows that $|p_1-p_2|=|p_2-p_3|=\cdots =|p_{n-1}-p_n|=2$ and the result follows immediately.
\end{proof}

\begin{remark}
Lemma \ref{key lemma} tells us that a good way to search for twin primes is to seek out for sequence of  consecutive primes whose value on the second prime-index function coincides. Next we extend this result to consecutive primes with gaps bigger than $2$. This extension also relates the second and the first prime index function to the gap of primes.
\end{remark}
\bigskip

\begin{lemma}\label{gensel4}
Let $p_{i}<p_{i+1}$ be any two consecutive primes such that $p_{i+1}-p_i=d$ for $d\geq 2$, then we have \begin{align}\xi(p_{i+1})=\xi(p_{i})+\iota(p_i)(d-2).\nonumber
\end{align}
\end{lemma}

\begin{proof}
First let $p_{i}<p_{i+1}$ such that $p_{i+1}-p_i=d$, then  it follows that the interval $(p_i,p_{i+1})$ is free of primes and it must certainly be that \begin{align}\xi(p_i+d-1)=\xi(p_i)+\iota(p_i)(d-1)\label{gensel}
\end{align}Again reinforcing the assumption that the primes are consecutive, then it follows by Definition \ref{first} and  \ref{gensel} \begin{align}\xi(p_{i+1})&=\xi(p_i+d-1)+\iota(p_{i+1})\nonumber \\&=\xi(p_i)+\iota(p_i)(d-1)+\iota(p_{i+1})\nonumber \\&=\xi(p_i)+\iota(p_i)(d-2)\nonumber
\end{align}thereby establishing the relation.
\end{proof}
\bigskip
Cramer's conjecture is the assertion that the gaps between consecutive primes grow at poly-logarithmic rate.  More precisely for any sequence of consecutive primes $p_n<p_{n+1}$, then we have \begin{align}p_{n+1}-p_n\ll (\log p_n)^2.\nonumber
\end{align}The conjecture can be expressed in terms of the primes index function. Recall the relation in Lemma \ref{gensel4}, then we have \begin{align}|\xi(p_{n+1})-\xi(p_n)|+2=d,\nonumber
\end{align}where $p_{n+1}-p_n=d$. Thus Cramer's conjecture can be restated as follows:
\bigskip

\begin{conjecture}~(Cramer)
Let $p_n<p_{n+1}$ be any sequence of consecutive primes. Then we have \begin{align}\xi(p_{n+1})-\xi(p_{n})\ll (\log p_n)^2.\nonumber
\end{align}
\end{conjecture} 
\bigskip

\begin{remark}
A good way of attacking Cramer's conjecture is to establish a good estimate for the second prime index function.
\end{remark}
\bigskip

It is known that there are infinitely many primes in arithmetic progression \cite{May}. Yet very little is known concerning their local distribution; that is their distribution in small intervals. Next we state a result which in essence is a consequence of Theorem \ref{key lemma}.
 
\begin{corollary}
Let $[a,b]\subset \mathbb{R}$. If $[a,b]$ contains finitely many consecutive primes $p_1<p_2<\ldots <p_n$ such that $\xi(p_1)=\xi(p_2)=\ldots =\xi(p_n)$, then $[a,b]$ contains primes in an arithmetic progression. 
\end{corollary}

\begin{proof}
Suppose $p_1<p_2<\cdots <p_n$ are primes in $[a,b]$ such that $\xi(p_1)=\xi(p_2)=\ldots =\xi(p_n)$. Then then the sequence of prime $p_1<p_2<\cdots <p_n$ must be an arithmetic progression, and the result follows immediately.
\end{proof}
\bigskip

The even Goldbach conjecture is another long-standing problem in mathematics. It predicts that every even number can be written as a sum of two primes. There are as many formulations of this problem, including quantitative versions \cite{iwaniec2004analytic}. We prove a result related to this conjecture, using the notion of oscillation at a point.

\begin{theorem}
Let $x\in \mathbb{N}$ be a point of oscillation of $\xi$ with period $t\in \mathbb{N}$. If $x-t$ is prime and the open interval $(x-t,x+t)$ contains no prime, then $2x$ can be written as a sum of two primes. 
\end{theorem}

\begin{proof}
Suppose $x\in \mathbb{N}$ is a point of oscillation of $\xi$ with period $t\in \mathbb{N}$. Then it follows by Definition \ref{oscillation} that $\xi(x-t)=\xi(x+t)$. Again it follows by Theorem \ref{bertrand} that there exist at least one prime in the interval $(x-t,x+t]$. Since there are no primes in the interval $(x-t,x+t)$, it follows that $x+t$ must neccessarily be prime. Since $x-t$ is prime, it follows that the representaion \begin{align}2x=(x-t)+(x+t)\nonumber
\end{align}is a partition into two primes, thereby ending the proof.
\end{proof}

\begin{remark}
Next we prove a result which tells us that points of oscillations with with period can never be prime.
\end{remark}

\begin{theorem}
Let $x$ be a point of oscillation of $\xi$ with period $1$. Then $x+1$ must necessarily be prime.
\end{theorem}

\begin{proof}
Suppose $x$ is a point of oscillation of $\xi$ with period $1$. Then it follows by Definition \ref{oscillation} that $\xi(x-1)=\xi(x+1)$. Again it follows by Theorem \ref{bertrand} that there must be at least one prime in the interval $(x-1,x+1]$. The obvious candidates must be either $x$ or $x+1$. We claim that $x+1$ must be prime. Suppose the contrary, that $x$ is prime. In the case $\xi(x-2)<\xi(x-1)$, then by letting $\xi(x-1)=L$, it follows that $\xi(x+1)=L-2$. This contradicts the fact that $x$ is a point of oscillation of $\xi$. On the other hand, let us assume $\xi(x-2)>\xi(x-1)=L$. Then it follows that $\xi(x+1)=L+2$, which again is absurd. This completes the proof.
\end{proof}

\section{Deviation in oscillations of $\xi$}
In this section we introduce the notion of deviation in oscillations of the second prime-index function. We relate this concept with the theory of prime gaps.

\begin{definition}\label{deviation}
Let $x,y\in \mathbb{R}$ with $x<y$. Then by the deviation in oscillation of $\xi$ at the points $x$ and $y$, we mean the value \begin{align}\mathcal{D}(\xi(x),\xi(y))=\bigg|\sum \limits_{\substack{n\geq 1\\x_{n}, x_{n+1}\in \mathbb{N}}}\xi(x_{n+1})-\xi(x_{n})\bigg|\nonumber
\end{align}for $x_n,x_{n+1}\in (x,y)$.
\end{definition}

\begin{theorem}\label{large deviation}
There exist some $N_0>0$ such that \begin{align}\mathcal{D}(\xi(p),\xi(q))>N\nonumber
\end{align}for all $N\geq N_0$, where $p$ and $q$ are primes.
\end{theorem}

\begin{proof}
Using the fact that there are arbitrarily large gaps between primes, choose $p,q\in \mathbb{\wp}$ such that $|p-q|=\mathrm{Inf}(|p-q|)>\delta$ for $\delta>0$ sufficiently large. It follows that for points $x_n,x_{n+1}\in (p,q)$, each ~$\xi(x_{n+1})-\xi(x_n)$ are of the same sign, and so the discrepancy \begin{align}\bigg|\sum \limits_{\substack{n\geq 1\\x_{n}, x_{n+1}\in \mathbb{N}}}\xi(x_{n+1})-\xi(x_{n})\bigg|\nonumber
\end{align}is large, thereby ending the proof.
\end{proof}

\begin{remark}
Theorem \ref{deviation} roughly speaking tells us that the deviation in oscillations of $\xi$ at any two points in $\mathbb{R}$ can be made arbitrarily large. Next we expose an equivalence between the distribution of twin primes and the prime index function. That is to say, the twin prime conjecture can be formulated and proven in this language if studied further.
\end{remark}

\begin{lemma}
Let $p>2$ be a prime, then $p$ is a twin prime if and only if $\xi(p)=\xi(p+2)$.
\end{lemma}

\begin{proof}
First let $p$ be prime, then it is clear that $\iota(p+1)=\iota(p)$ since $p+1$ cannot be prime. It follows that \begin{align}\xi(p+1)&=\xi(p)+\iota(p+1)\nonumber \\&=\xi(p)+\iota(p).\label{gensel1}
\end{align}Again\begin{align}\xi(p+2)=\xi(p+1)+\iota(p+2).\label{gensel2}
\end{align}Combining \ref{gensel1} and \ref{gensel2} with the underlying assumption that $\xi(p)=\xi(p+2)$, we obtain the relation $\iota(p)=-\iota(p+2)$. This certainly implies that $p+2$ is prime. Thus $p$ must indeed be a twin prime. Conversely suppose $p$ is a twin prime, then $p+2$ is also prime. By applying the relation  in Lemma \ref{gensel4} and setting $d=2$, the result follows immediately.
\end{proof}

\section{End remarks}
The twin prime conjecture \cite{tenenbaum2015introduction} is one of the oldest unsolved problem in mathematics. Thanks to the recent progress  towards the complete resolution of the conjecture. It states that there are infinitely many prime pairs $(p,q)$ with $p\neq q$ such that $|p-q|=2$. In other words there are infinitely many intervals of length $2$ that contains two primes. In relation to Lemma \ref{gensel2}, the twin prime conjecture can be reformulated as 
\bigskip

\begin{conjecture}~(Twin prime)
There are infinitely many consecutive primes $p_1<p_2$ such that $\xi(p_1)=\xi(p_2)$.
\end{conjecture} 
\bigskip

In a similarly vein, Cramer's conjecture can be stated in terms of the second prime as 
\bigskip

\begin{conjecture}~(Cramer)
Let $p_{n+1}>p_{n}$ be any sequence of consecutive primes, then we have \begin{align}\xi(p_{n+1})-\xi(p_n)\ll (\log p_n)^2.\nonumber
\end{align}
\end{conjecture} 
%

\bibliographystyle{amsplain}

\begin{thebibliography}{10}




\bibitem {tenenbaum2015introduction} G{\'e}rald Tenenbaum, \textit{Introduction to analytic and probabilistic number theory}, vol. 163, American Mathematical Soc.,  2015.

\bibitem {May}Montgomery, H.L, and  Vaughan, R.C, \textit{Multiplicative number theory 1}:Classical theory. vol.97, 
Cambridge university press, 2006.

\bibitem {hildebrand2005introduction}Hildebrand, AJ,  \textit{Introduction to Analytic Number Theory Lecture Notes}, Department of Mathematics, University of Illinois, 2005.


\bibitem {iwaniec2004analytic}H. Iwaniec, E.Kowalski  \textit{Analytic number theory, volume 53 of American Mathematical Society Colloquium Publications.}, American Mathematical Society, Providence, RI, vol. 53, 2004, pp.4.

\end{thebibliography}

\end{document}